\documentclass[a4paper,12pt]{article}

\usepackage[utf8]{inputenc}
\usepackage[T1]{fontenc}
\usepackage[english]{babel}

\usepackage{amsmath,amsfonts,amssymb,amsthm}
\usepackage{graphicx,tabularx,color,subfigure}
\usepackage[all]{xy}
 
\newtheorem{defn}{Definition}
\newtheorem{lem}{Lemma}
\newtheorem{thm}{Theorem}
\newtheorem{pro}{Proposition}
\newtheorem{rem}{Remark}

\newcommand{\C}{\mathbb{C}}
\newcommand{\CC}{\widehat{\mathbb{C}}}
\newcommand{\D}{\mathbb{D}}
\newcommand{\Z}{\mathbb{Z}}
\newcommand{\N}{\mathbb{N}}

\newcommand{\Julia}{\mathcal{J}}
\newcommand{\Fatou}{\mathcal{F}}
\newcommand{\HPcFP}{\mathrm{HPcFP}}
\newcommand{\Postcrit}{\mathcal{P}}
\newcommand{\Pole}{\mathcal{D}}
\renewcommand{\outer}{\mathrm{out}}
\newcommand{\inner}{\mathrm{in}}

\newcommand{\annulus}[2]{A\left(#1,#2\right)}
\newcommand{\modulus}[2]{\mathrm{mod}\left(#1,#2\right)}

\newtheorem*{thmA}{Theorem A}

\title{On McMullen-like mappings}
\author{\small Antonio Garijo \\
{\small Dept.~d'Eng.~Inform\`atica i Matem\`atiques}\\
{\small Universitat Rovira i Virgili}\\
{\small Av. Pa\"isos Catalans 26}\\
{\small Tarragona 43007, Spain} \and 
{\small S\'ebastien Godillon}\\ 
{\small Dept.~de Matem\`atica Aplicada i An\`alisi}\\
{\small Universitat de Barcelona}\\
{\small Gran Via de les Corts Catalanes, 585}\\
{\small 08007 Barcelona, Spain} }
\date{March 10, 2014}

\begin{document}


\maketitle

\begin{abstract}
{We introduce a generalization of the McMullen family  $f_{\lambda}(z)=z^n + \lambda/z^d$. In 1988 C. McMullen showed that the Julia set of $f_{\lambda}$ is a Cantor set of circles if and only if $1/n+1/d <1$ and the simple critical values of $f_\lambda$ belong to the trap door. We generalize this behavior and we define  a McMullen-like mapping as a rational map $f$  associated to a hyperbolic postcritically finite  polynomial $P$ and a pole data $\Pole$ where we 
encode, basically, the location of every pole of $f$ and the local degree at each pole. In the McMullen family the polynomial $P$ is $z \mapsto z^n$ and the pole data $\Pole$ is the pole located at the origin that maps to infinity with local degree $d$. As in the McMullen family $f_{\lambda}$, we can characterize a McMullen-like mapping using an arithmetic condition depending only on the polynomial $P$ and the pole data $\Pole$. We prove that the arithmetic condition is necessary using the theory of Thurston's obstructions, and sufficient by quasiconformal surgery.}
\end{abstract}

\tableofcontents

\newpage


\section{Introduction}\label{sec:intro}

The {\it Fatou set} of a rational map $f$, denoted by $\Fatou(f)$, is defined to be the set of points at which the family of iterates of $f$ is a normal family in the sense of Montel. The complement of the Fatou set, in the Riemann sphere $\CC$, is the {\it Julia set}  and is denoted by $\Julia(f)$. By definition, the Fatou set is an open set and the the Julia set is closed. The Julia set is also the closure of the set of repelling periodic points of $f$, and it is the set where $f$ has sensitive dependence on initial conditions. Both sets, $\Fatou(f)$ and $\Julia(f)$, are completely invariant. Equivalently, the Julia set $\Julia(f)$ is the smallest closed set containing at least three points which is completely invariant under $f$.  For a deep and helpful introduction on iteration of rational maps see \cite{BeardonBook,CarlesonGamelinBook,MilnorBook,SteinmetzBook}.  One of the goals in complex dynamics are to study the topological properties of the Julia and Fatou sets of $f$ and dynamics of $f$ restricted to these sets.

In 1988 C. McMullen (\cite{AutomorphismsRationalMaps}) showed the first example of a rational map whose Julia set is a Cantor set of circles, the rational map that exhibits this phenomenon, hereafter McMullen family,  is $f_{\lambda}(z)=z^n+\lambda/z^d$  for some values of $n,d\in\N$ and $\lambda\in\C$. This kind of rational maps could be viewed as a singular perturbation of the polynomial  $z\mapsto z^n$ when we add a pole at the origin of order $d$. McMullen family has focused the attention for several reasons.  On one hand they exhibits classical Julia sets including  Cantor sets, Sierpi\'nski curves, and Cantor sets of circles (\cite{Trichotomy}) and on the other hand the parameter space has complex dimension one, since the free critical points of $f_{\lambda}$ behave in a symmetric way.  The McMullen family has been studied extensively by R. Devaney et al. (\cite{Trichotomy, EvolutionMcMullenDomain, SingularPerturbationsMcMullenDomain}), N. Steinmetz (\cite{DynamicsMcMullenFamily}) and Qiu W., P. Roesch, Wang X. and Yin Y. (\cite{DynamicsMcMullenMaps, HyperbolicComponentsMcMullenMaps}) among others. We refer to \cite{SingularPerturbationsComplexPolynomials} for a survey of the main results about singular perturbations of complex polynomials and references therein.

For the polynomial $z\mapsto z^n$, with $n\geqslant 2$,  infinity
and the origin are super-attracting fixed points and the Julia set is the unit
circle. When we add the perturbation $\lambda/z^d$  in the McMullen family, 
several aspects of the dynamics remain the same, but
others change dramatically.  For example, when $\lambda \neq 0$, the
point at $\infty$ is still a super-attracting fixed point and there
is an immediate basin of attraction of $\infty$ that we call $V_{\infty}$. On the other hand, there is a neighborhood of the pole located at the origin
that is now mapped $d$-to-1 onto a neighborhood of $\infty$. When this neighborhood is
disjoint from $V_{\infty}$ we call it the \emph{trap door} and denote it by
$T_0$. Every point that escapes to infinity and does not lie in
$V_{\infty}$ has to do so by passing through $T_0$. Since the degree of $f_{\lambda}$
changes from $n$ to $n+d$, some additional critical points are
created. The set of critical points includes $\infty$ and $0$ whose
orbits are completely determined, so there are $n+d$ additional
``free'' critical points.  The orbits of these points are of
fundamental importance in characterizing the Julia set of $f_{\lambda}$.  

McMullen showed that if the arithmetic condition $1/n+1/d<1$ is satisfied and the free critical values lie in the trap door $T_0$ then the 
Julia set of $f_{\lambda}$ is a Cantor set of circles (\cite{AutomorphismsRationalMaps,Trichotomy}). Under these assumptions we notice that the $n+d$ free critical points belong 
to a doubly connected  Fatou component $A$  separating the trap door $T_0$ and the immediate basin of infinity $V_{\infty}$ and such that $f_{\lambda}$ maps 
the annulus $A$ onto the trap door $T_0$ with degree $n+d$. We also notice that the minimal degree of $f_{\lambda}$ satisfying this arithmetic condition is 5.

In the last years they have appeared in the literature several woks dealing with  singular perturbations of polynomials of the form $P_{c}(z)=z^n+c$ where $c$ is chosen to be the center of a hyperbolic component of the corresponding Multibrot set  and adding a perturbation with one or several poles (see \cite{GeneralizedMcMullenDomain,SingularPerturbationsQuadraticMultiplePoles}). 
Figure~\ref{fig:mcmullen_ex1} display the Julia set of three singular perturbation of polynomials. The first one corresponds to a member of the McMullen family, concretely the rational map $z^3-0.01/z^3$. The second one is the Julia set of the rational map $z^3+i-10^{-7}/z^3$, that corresponds to a singular perturbation of the polynomial $z^3+i$, that exhibits a super-attracting cycle $0 \mapsto i \mapsto 0$, when we add a pole at the origin. Finally, the third one is the Julia set of the rational map $z^2-1+10^{-22}/(z^7(z+1)^5)$ that corresponds to a perturbation of the quadratic polynomial $z^2-1$, with super-attracting cycle $0 \mapsto -1 \mapsto 0$,  when we add two poles, one at $z=0$ and another one  at $z=-1$.  In this figure we also show the dynamical plane of the corresponding polynomial and we mark the Fatou domain with a number where we add a pole with corresponding local degree.

\begin{figure}[ht]
	\centering
	\includegraphics[width=0.8\textwidth]{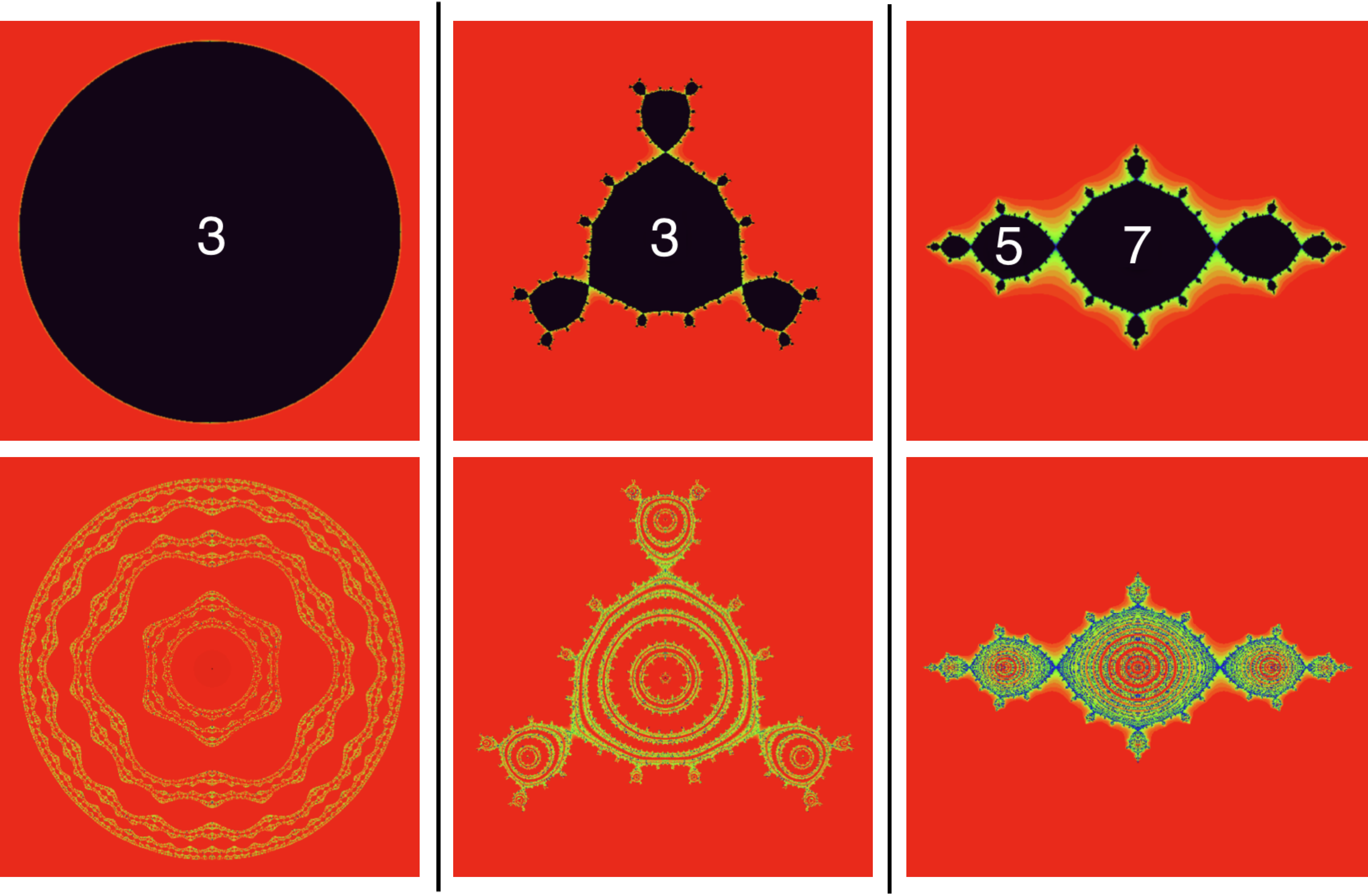}
	\setlength{\unitlength}{220pt}		  
	   \caption{\small{Three examples of McMullen-like mappings. In the left, we show the dynamical plane of the polynomial $z^3$ (up) and the rational map $z^3 -0.01/z^3$ (down). In the middle, the cubic polynomial $z^3+i$ and the rational map $z^3+i-10^{-7}/z^3$. And finally in the right, the quadratic polynomial $z^2-1$ and the rational map $z^2-1 + 10^{-22}/(z^7(z+1)^5)$. In each case, we mark with the local degree the bounded Fatou component where we put a pole.}}    
	\label{fig:mcmullen_ex1}
\end{figure}

The main goal of this paper is to present an unified approach to this kind of dynamical systems. We firstly define what we call a {\it McMullen-like mapping}. Here we give an idea  and we refer to the next section for a precise definition of it. Before the definition we need two ingredients:  the first one is a  hyperbolic postcritically finite polynomial $P$ (hereafter $\HPcFP$)  and the second one is the pole data $\Pole$ that encodes the information about the poles. 

A McMullen-like mapping of type $(P,\Pole)$, formed by a $\HPcFP$ $P$ and a pole data $\Pole$, is a rational map $f$ verifying the following conditions. The first condition is that $\infty$ is a super-attracting fixed point of $f$. We denote by $V_{\infty}$ the immediate  basin of attraction of infinity and we require that $\partial V_{\infty}$ is a homeomorphic copy of the Julia set of the polynomial $P$. The second condition is about the trap doors, these trap doors are simply connected domains and the rational map $f$ sends  every trap door onto $V_{\infty}$ according to the pole data. Moreover, around every trap door we require the existence of an annulus $A$ (containing some critical points) and such that $f$ sends $A$ onto a simply connected domain. This simply connected domain is mapped in a similar way as the polynomial $P$ until it reaches the close trap door. And finally, outside the components that appears in the pole data we basically apply the  dynamics of the polynomial $P$. 

The main result of this paper is about the existence of McMullen-like mappings. More precisely, given a pair $(P,\Pole)$, formed by a $\HPcFP$ $P$ and a pole data $\Pole$, we can characterize the existence of a McMullen-like mapping of type $(P,\Pole)$ under an arithmetic condition.  We also describe the Julia and the Fatou set of any McMullen-like mapping. 

We organize the rest of the paper in the following way. In Section \ref{sec:def} we give the precise definition of a McMullen-like mapping, showing some examples from the literature of this kind of rational maps. We also describe the Julia set of a McMullen-like mapping. In Section \ref{sec:existence} we state Theorem A which is the main result of this work. In Theorem A we characterize the existence of  a McMullen-like mapping of type $(P,\Pole)$ using an arithmetic condition depending on the polynomial $P$ and the pole data $\Pole$.  The rest of Section \ref{sec:existence} is devoted to prove Theorem A. In order to prove the necessity  of this arithmetic condition  (Subsection \ref{subsec:necessity}) we use the theory developed by W. Thurston about obstructions (see \cite{ThurstonProof, McMullenBook}). Then using quasiconformal surgery (see \cite{AhlforsQuasiconformalBook,QuasiconformalSurgeryBook}) we are able to construct a McMullen-like mapping of type $(P,\Pole)$ assuming only the arithmetic condition (Subsection \ref{subsec:construction}), proving thus the sufficiency of this arithmetic condition. Finally, in Section \ref{sec:new_examples} we show some new examples  
of McMullen-like mappings and we give an explicit expression of a McMullen-like mapping of minimal degree 4.

\bigskip
\noindent \emph{Acknowledgments.} The first author are partially supported by the Catalan
grant 2009SGR-792 and by the Spanish grant MTM-2008-01486 Consolider
(including a FEDER contribution).


\section{Definition of a McMullen-like mapping and properties}\label{sec:def}

\subsection{Definition of a McMullen-like mapping}\label{subsec:definition}

Let $P$ be a hyperbolic postcritically finite  polynomial, we shall write $\HPcFP$ in short, namely a polynomial map $P:\C\rightarrow\C$ of degree $n=\deg(P)\geqslant 2$ such that every critical point is eventually mapped under iteration to a super-attracting periodic cycle. The Julia set $\Julia(P)$ is connected  and the Fatou set $\Fatou(P)$ contains countably many connected components, called Fatou domains, which are simply connected. Moreover the unbounded Fatou domain, denoted by $U_{\infty}$, is a completely  invariant super-attracting basin with $\partial U_{\infty}=\Julia(P)$, and every bounded Fatou domain $U$ is eventually mapped to a periodic cycle of immediate super-attracting basins (see \cite{BeardonBook,CarlesonGamelinBook,MilnorBook,SteinmetzBook}).

We arbitrary chose a labeling of the finitely many periodic bounded Fatou domains, of the following form
$$\{U_{i,j}\ /\ 1\leqslant i\leqslant N\ \text{and}\ j\in\Z/p_{i}\Z\}$$
where $ N\geqslant 1$ is the number of bounded super-attracting periodic cycles, $p_{i}\geqslant 1$ is the period of the $i$-th cycle, and so that $P(U_{i,j})=U_{i,j+1}$ for every $1\leqslant i\leqslant N$ and $j\in\Z/p_{i}\Z$. Moreover for every periodic bounded Fatou domain $U_{i,j}$, we denote by $n_{i,j}$ the degree of the restriction $P|_{U_{i,j}}$, which coincides with the degree of the restriction $P|_{\partial U_{i,j}}$. Notice that the Riemann-Hurwitz formula gives $n-1=\sum_{i,j}(n_{i,j}-1)$.

\begin{defn}[Pole data]\label{DefPoleData}
A pole data $\Pole$ associated to a $\HPcFP$ $P$ is a nonempty collection of periodic bounded Fatou domains of $P$, each provided with a positive integer. More precisely, $\Pole$ is the data of a nonempty subset of $\{U_{i,j}\ /\ 1\leqslant i\leqslant N\ \text{and}\ j\in\Z/p_{i}\Z\}$ and a function from this subset to $\N\setminus\{0\}$ denoted by $U_{i,j}\mapsto d_{i,j}$.

With abuse of notation, we write $U_{i,j}\in\Pole$ if and only if $U_{i,j}$ is picked in the pole data $\Pole$, and conversely $U\notin\Pole$ for every bounded Fatou domain $U$ which is not a $U_{i,j}\in\Pole$.

The degree of the pole data $\Pole$ is defined to be $\deg(\Pole):=d=\sum_{U_{i,j}\in\Pole}d_{i,j}\geqslant 1$.
\end{defn}

Remark that if a rational map $f:\CC\rightarrow\CC$ and a simply connected domain $V_{\infty}\subset\CC$ are so that $f|_{\partial V_{\infty}}$ is topologically conjugate to $P|_{\partial U_{\infty}}$, namely there exists an orientation preserving homeomorphism $\varphi:\partial U_{\infty}\rightarrow\partial V_{\infty}$ so that $f\circ\varphi=\varphi\circ P$ on $\partial U_{\infty}$, then we may define for every bounded Fatou domain $U$ of $P$ a simply connected domain $V(U)\subset\CC$ as the unique connected component of $\CC\setminus\varphi(\partial U)$ which does not contain $V_{\infty}$ (it is well defined because $\partial U$ is a Jordan curve since $P$ is hyperbolic, see for instance Lemma 19.3 in \cite{MilnorBook}). In order to lighten the notation, we write $V_{i,j}=V(U_{i,j})$ for every $1\leqslant i\leqslant N$ and $j\in\Z/p_{i}\Z$.

\begin{defn}[McMullen-like mapping]\label{DefMcMullenLikeMap}
Let $\Pole$ be a pole data associated to a $\HPcFP$ $P$.
A McMullen-like mapping of type $(P,\Pole)$, is a rational map $f:\CC\rightarrow\CC$ such that the following holds.
\begin{description}
	\item[(i)] There exists a simply connected domain $V_{\infty}\subset\CC$ such that $f(V_{\infty})=V_{\infty}$, and $f|_{\partial V_{\infty}}$ is topologically conjugate to $P|_{\partial U_{\infty}}$.
	\item[(ii)] For every $U\notin\Pole$, $f(V(U))=V(P(U))$.
	\item[(iii)] For every $U_{i,j}\in\Pole$, there exist a simply connected domain $T_{i,j}\subset V_{i,j}$, called a trap door, and a doubly connected domain $A_{i,j}\subset V_{i,j}\setminus\overline{T_{i,j}}$ which separates $\partial V_{i,j}$ from $\overline{T_{i,j}}$ such that
	\begin{itemize}
		\item $f(T_{i,j})=V_{\infty}$ and $f|_{T_{i,j}}$ has degree $d_{i,j}$;
		\item $f(A_{i,j})$ is a simply connected domain contained in $V(P(U_{i,j}))=V_{i,j+1}$ and $f|_{A_{i,j}}$ is a proper map (namely $f(\partial A_{i,j})=\partial f(A_{i,j})$);
		\item $f$ has no critical points in $\overline{V_{i,j}}\setminus(A_{i,j}\cup T_{i,j})$.
	\end{itemize}
	\item[(iv)] For every critical point $c$ of $f$, if $c$ is eventually mapped under iteration of $f$ into $V_{i,j}$ for some $U_{i,j}\in\Pole$, or equivalently if
	$$t_{c}=\min\{k\geqslant 1\ /\ \exists U_{i,j}\in\Pole,\ f^{k}(c)\in V_{i,j}\}<+\infty,$$
	then $c$ is mapped into the corresponding trap door, namely $f^{t_{c}}(c)\in T_{i,j}$.
\end{description}
\end{defn}

According to the points \textbf{(i)} and \textbf{(ii)} above, $f|_{V_{\infty}}:V_{\infty}\rightarrow V_{\infty}$ is a holomorphic branched covering of degree $n=\deg(P)$ with $n-1$ critical points counted with multiplicity, and for every $U\notin\Pole$, $f|_{V(U)}:V(U)\rightarrow V(P(U))$ is a holomorphic branched covering of same degree as $P|_{U}$ with the same number of critical points counted with multiplicity. The following lemma extends the point \textbf{(iii)} above by describing how $f$ acts on each $V_{i,j}$.

\begin{lem}\label{LemDynamicsOnVij}
For every $U_{i,j}\in\Pole$, denote by $\overline{A_{i,j}^{\outer}}$ the closed annulus between $\partial V_{i,j}$ and $\partial A_{i,j}$, and by $\overline{A_{i,j}^{\inner}}$ the closed annulus between $\partial A_{i,j}$ and $\partial T_{i,j}$. Then the action of $f$ on $\overline{V_{i,j}}=\overline{A_{i,j}^{\outer}}\cup A_{i,j}\cup\overline{A_{i,j}^{\inner}}\cup T_{i,j}$ is as follows:
\begin{itemize}
	\item $\overline{f(A_{i,j}^{\outer})}$ is the closed annulus in $\overline{V_{i,j+1}}$ between $\partial V_{i,j+1}$ and $\partial f(A_{i,j})$, and $f|_{A_{i,j}^{\outer}}$ is a holomorphic covering of degree $n_{i,j}$;
	\item $f(A_{i,j})$ is a simply connected domain in $V_{i,j+1}$, and $f|_{A_{i,j}}$ is a holomorphic branched covering of degree $n_{i,j}+d_{i,j}$ with $n_{i,j}+d_{i,j}$ critical points counted with multiplicity;
	\item $\overline{f(A_{i,j}^{\inner})}$ is the closed annulus between $\partial f(A_{i,j})$ and $\partial V_{\infty}$, and $f|_{A_{i,j}^{\inner}}$ is a holomorphic covering of degree $d_{i,j}$;
	\item $f(T_{i,j})=V_{\infty}$, and $f|_{T_{i,j}}$ is a holomorphic branched covering of degree $d_{i,j}$ with $d_{i,j}-1$ critical points counted with multiplicity.
\end{itemize}
In particular, $\overline{f(V_{i,j})}=\CC$ (and hence $f|_{V_{i,j}}$ is not a proper map).
\end{lem}

\begin{proof}
Since $f$ has no critical points in $\overline{V_{i,j}}\setminus(T_{i,j}\cup A_{i,j})$, it follows that $f|_{A_{i,j}}$ has degree
$$\deg(f|_{A_{i,j}})=\deg(f|_{\partial V_{i,j}})+\deg(f|_{T_{i,j}})=n_{i,j}+d_{i,j}$$
because $\deg(f|_{\partial V_{i,j}})=\deg(P|_{\partial U_{i,j}})=n_{i,j}$ from the point \textbf{(ii)} in Definition \ref{DefMcMullenLikeMap}. The remaining easily follows by using the Riemann-Hurwitz formula.
\end{proof}

Notice that each of the critical points $c\in A_{i,j}$ satisfies $f(c)\in V_{i,j+1}$, thus $t_{c}\leqslant p_{i}<+\infty$ and $t_{c}$ only depends on $A_{i,j}$, namely on $U_{i,j}\in\Pole$. The following lemma shows that $A_{i,j}$ may be chosen in Definition \ref{DefMcMullenLikeMap} in order that the whole image $f^{t_{c}}(A_{i,j})$ is actually the trap door containing $f^{t_{c}}(c)$.

\begin{lem}\label{LemPreimageTrapDoor}
For every $U_{i,j}\in\Pole$, we may assume without loss of generality that $f^{t_{i,j}}(A_{i,j})=T_{i,j+t_{i,j}}$ where
$$\ t_{i,j}=\min\{k\geqslant 1\ /\ U_{i,j+k}\in\Pole\}.$$
\end{lem}

This result is similar to the first proposition of Section 3 in \cite{Trichotomy} for the McMullen mapping $z\mapsto z^{n}+\frac{\lambda}{z^{d}}$, except we can not use the symmetry of the map here.

\begin{proof}
Remark that $f^{t_{i,j}-1}(V_{i,j+1})=V_{i,j+t_{i,j}}$ from the definition of $t_{i,j}$ and the point \textbf{(ii)} in Definition \ref{DefMcMullenLikeMap}, and the critical values of $f^{t_{i,j}-1}|_{V_{i,j+1}}$ lie in $T_{i,j+t_{i,j}}$ from the point \textbf{(iv)}. Therefore the preimage of $T_{i,j+t_{i,j}}$ under $f^{t_{i,j}-1}|_{V_{i,j+1}}$ is a simply connected domain in $V_{i,j+1}$.

Let $A'_{i,j}$ be the set of points $z\in V_{i,j}$ such that $f(z)\in(f^{t_{i,j}-1}|_{V_{i,j+1}})^{-1}(T_{i,j+t_{i,j}})$. Since $f(V_{i,j})=\CC$ from Lemma \ref{LemDynamicsOnVij}, it follows that $f(A'_{i,j})$ is a simply connected domain in $V_{i,j+1}$, $f|_{A'_{i,j}}$ is a proper map, and $f^{t_{i,j}}(A'_{i,j})=T_{i,j+t_{i,j}}$. Notice that $A'_{i,j}\subset V_{i,j}\setminus\overline{T_{i,j}}$ because $f(T_{i,j})=V_{\infty}$ is outside $V_{i,j+1}$. Furthermore, every critical point $z\in A_{i,j}$ satisfies $f^{t_{i,j}}(c)\in T_{i,j+t_{i,j}}$ from the definition of $t_{i,j}$ and the point \textbf{(iv)}, and hence is in $A'_{i,j}$. Therefore $f$ has no critical points in $\overline{V_{i,j}}\setminus(A'_{i,j}\cup T_{i,j})$

Consequently, it is enough to show that $A'_{i,j}$ is a doubly connected domain which separates $\partial V_{i,j}$ from $\overline{T_{i,j}}$. At first remark that $f|_{A'_{i,j}}$ is actually a holomorphic branched covering of degree $\deg(f|_{A'_{i,j}})=n_{i,j}+d_{i,j}$ with $\nu(f|_{A'_{i,j}})=n_{i,j}+d_{i,j}$ critical points counted with multiplicity from Lemma \ref{LemDynamicsOnVij} since every critical point and every cocritical point of $f|_{A_{i,j}}$ is in $A'_{i,j}$ by definition. Denote by $\{C_{\ell}\ /\ 1\leqslant \ell\leqslant L\}$ the collection of $L\geqslant 1$ connected components of $A'_{i,j}$, and by $m_{\ell}\geqslant 1$ the number of connected components in every $\partial C_{\ell}$. The Riemann-Hurwitz formula gives
$$\forall 1\leqslant\ell\leqslant L,\quad 2-m_{\ell}=(2-1)\deg(f|_{C_{\ell}})-\nu(f|_{C_{\ell}})$$
that leads by summing to
\begin{equation}\label{EqLemPreimageTrapDoor}
	2L-\sum_{\ell=1}^{L}m_{\ell}=\sum_{\ell=1}^{L}\deg(f|_{C_{\ell}})-\sum_{\ell=1}^{L}\nu(f|_{C_{\ell}})=\deg(f|_{A'_{i,j}})-\nu(f|_{A'_{i,j}})=0.
\end{equation}
Furthermore, every simply connected components of $V_{i,j}\setminus\overline{A'_{i,j}}$ must contain some points which are mapped into $V_{\infty}\subset\CC\setminus\overline{f(A'_{i,j})}$, and hence must contain $T_{i,j}$ from Lemma \ref{LemDynamicsOnVij}. In particular there is no $m_{\ell}\geqslant 3$, and (\ref{EqLemPreimageTrapDoor}) implies that every $C_{\ell}$ is a doubly connected domain which separates $\partial V_{i,j}$ from $\overline{T_{i,j}}$. Assume the $(C_{\ell})_{1\leqslant \ell\leqslant L}$ are ordered labelled from $\partial V_{i,j}$ to $\overline{T_{i,j}}$, then $\deg(f|_{C_{1}})\geqslant n_{i,j}+1$ and $\deg(f|_{C_{L}})\geqslant 1+d_{i,j}$ lead to a contradiction as soon as $L\geqslant 2$.
\end{proof}

As a consequence, the orbit of $A_{i,j}$ is as follows.
$$\begin{array}{rrclcl}
	& A_{i,j} & \subset & V_{i,j} & \text{with} & U_{i,j}\in\Pole \\
	\forall 1\leqslant k<t_{i,j}, & f^{k}(A_{i,j}) & \subset & V_{i,j+k} & \text{with} & U_{i,j+k}\notin\Pole \\
	& f^{t_{i,j}}(A_{i,j}) & = & T_{i,j+t_{i,j}}\subset V_{i,j+t_{i,j}} & \text{with} & U_{i,j+t_{i,j}}\in\Pole \\
	\forall k>t_{i,j}, & f^{k}(A_{i,j}) & = & V_{\infty} && \\
\end{array}$$


\subsection{The Julia set of a McMullen-like mapping}

The following result describes the Julia set of any McMullen-like mapping.

\begin{thm}\label{ThmJuliaComponents}
If $f$ is a McMullen-like mapping of type $(P,\Pole)$ for some pole data $\Pole$ associated to a $\HPcFP$ $P$, then $f$ is a hyperbolic rational map of degree $\deg(f)=\deg(P)+\deg(\Pole)$, the Julia set $\Julia(f)$ of $f$ is disconnected, and the Fatou set $\Fatou(f)$ contains infinitely many connected components which are either simply or doubly connected. Moreover $\Julia(f)$ contains
\begin{itemize}
	\item countably many preimages of $\partial V_{\infty}$ which is a fixed Julia component quasisymetrically equivalent to $\partial U_{\infty}=\Julia(P)$;
	\item countably many Cantor of Jordan curves so that each Jordan curve belongs to a different Julia component;
	\item and, if $P$ is not affine conjugate to $z\mapsto z^{n}$, uncountably many point Julia components which accumulate everywhere on $\Julia(f)$.
\end{itemize}
\end{thm}

There is actually much more structure in the Julia set $\Julia(f)$. Indeed, every ``unburied'' Jordan curve is eventually mapped under iteration onto a proper subset of $\partial V_{\infty}$. Therefore every Julia component which contains such a Jordan curve is actually quasisymetrically equivalent to a finite covering of $\partial U_{\infty}=\Julia(P)$, and hence comes with infinitely many ``decorations'' attached to the Jordan curve, provided $P$ is not affine conjugate to $z\mapsto z^{n}$ (this structure has already been noticed in \cite{SingularPerturbations} and \cite{SingularPerturbationsQuadraticMultiplePoles}). We will see in the proof below that except for the preimages of $\partial V_{\infty}$, all others Julia components are either points or Jordan curves (in particular, every buried Julia component is either a point or a Jordan curve, compare with \cite{FamilyBuriedJuliaComponents}), that provides a complete topological description of all Julia components.

\begin{proof}
We firstly compute the degree of the rational map $f$. The Riemann-Hurwitz formula gives $2\deg(f)-2=\nu(f)$ where $\nu(f)$ denotes the number of critical points of $f$ counted with multiplicity. From the definition of a McMullen-like mapping (Definition \ref{DefMcMullenLikeMap}), we have
$$\nu(f)=\nu(f|_{V_{\infty}})+\nu'(f)+\sum_{U_{i,j}\in\Pole}\nu(f|_{A_{i,j}})+\sum_{U_{i,j}\in\Pole}\nu(f|_{T_{i,j}})$$
where
\begin{itemize}
	\item $\nu(f|_{V_{\infty}})=n-1$ is the number of critical points of $f|_{V_{\infty}}$ counted with multiplicity;
	\item $\nu'(f)$ is the number of critical points of $f$ counted with multiplicity which are neither in $V_{\infty}$ nor in $\bigcup_{U_{i,j}\in\Pole}V_{i,j}$, namely
	$$\nu'(f)=\sum_{U\notin\Pole}(\deg(f|_{V(U)})-1)=\sum_{U\notin\Pole}(\deg(P|_{U})-1)=(n-1)-\sum_{U_{i,j}\in\Pole}(n_{i,j}-1);$$
	\item $\nu(f|_{A_{i,j}})=n_{i,j}+d_{i,j}$ is the number of critical points of $f|_{A_{i,j}}$ counted with multiplicity;
	\item $\nu(f|_{T_{i,j}})=d_{i,j}-1$ is the number of critical points of $f|_{T_{i,j}}$.
\end{itemize}
Putting everything together leads to
$$\deg(f)=\frac{1}{2}(\nu(f)+2)=\frac{1}{2}\left(2n+\sum_{U_{i,j}\in\Pole}2d_{i,j}\right)=n+d=\deg(P)+\deg(\Pole).$$

In order to prove that $f$ is a hyperbolic map we study the orbit of every critical point. Let $c$  be a critical point of $f$. From Definition \ref{DefMcMullenLikeMap} and Lemma \ref{LemPreimageTrapDoor}, one of the following holds:
\begin{itemize}
	\item either $c\in V_{\infty}$ (and hence $t_{c}=+\infty$), where $V_{\infty}$ is a fixed simply connected domain, that corresponds to a fixed immediate attracting or super-attracting basin;
	\item or $c\in T_{i,j}$ for some $U_{i,j}\in\Pole$ (and hence $t_{c}=+\infty$), where $T_{i,j}$ is a simply connected preimage by $f$ of $V_{\infty}$;
	\item or $c\in A_{i,j}$ for some $U_{i,j}\in\Pole$ (and hence $t_{c}=t_{i,j}$), where $A_{i,j}$ is a doubly connected preimage by $f^{t_{i,j}+1}$ of $V_{\infty}$;
		\item or $c\in V(U)$ for some $U\notin\Pole$ and $t_{c}<+\infty$, then $c$ lies in a simply connected preimage by $f^{t_{c}+1}$ of $V_{\infty}$;
		\item or $c\in V(U)$ for some $U\notin\Pole$ and $t_{c}=+\infty$, then $V(U)$ is a simply connected domain which is eventually mapped onto a periodic cycle of simply connected domains of the form $V(U_{i,j})$ with $U_{i,j}\notin\Pole$, that corresponds to a periodic cycle of immediate attracting or super-attracting basins.
\end{itemize}

Therefore the rational map $f$ is hyperbolic, and every Fatou domain is either simply or doubly connected. 

After the description of the Fatou set $\Fatou(f)$ we focus on the Julia set $\Julia(f)$. The Julia set $\Julia(f)$ is disconnected since there is at least one doubly connected Fatou domain. From the first property  in the definition of a McMullen-like mapping (Definition \ref{DefMcMullenLikeMap}), $\partial V_{\infty}$ is a fixed Julia component homeomorphic to $\partial U_{\infty}$. Using the surgery procedure described in \cite{AutomorphismsRationalMaps} (Sections 5 and  6), we can extend the homeomorphism to a quasiconformal conjugation on a neighborhood of $\partial U_{\infty}$, and hence $\partial V_{\infty}$ is quasisymetrically equivalent to $\partial U_{\infty}=\Julia(P)$.

Fix $U_{i,j}\in\Pole$. Consider the orbit of the closed annulus $\overline{A}=\overline{A_{i,j}^{\outer}}\cup A_{i,j}\cup\overline{A_{i,j}^{\inner}}$. From Definition \ref{DefMcMullenLikeMap} and Lemma \ref{LemDynamicsOnVij}, it turns out that $\overline{f^{p_{i}}(A)}$ covers $V_{i,j}$. The preimage of $\overline{A}$ by $f^{p_{i}}$ contains two disjoint closed annuli both nested in $\overline{A}$ (more precisely, they are nested in $\overline{A_{i,j}^{\outer}}$ and $\overline{A_{i,j}^{\inner}}$ respectively) which do not contain critical points of $f^{p_{i}}$. It is then straightforward to show (see \cite{AutomorphismsRationalMaps} or \cite{Trichotomy}) that the non-escaping set $\overline{\bigcap_{k\geqslant 1}(f^{p_{i}})^{-k}(A)}\subset\Julia(f)$ is homeomorphic to a Cantor of Jordan curves, namely to $\Sigma_{2}\times\partial\D$ where $\Sigma_{2}=\{0,1\}^{\N}$ is the set of all one-sided sequences on two symbols. Moreover, the topological dynamics is conjugate to the following skew product
$$(\varepsilon_{0}\varepsilon_{1}\varepsilon_{2}\dots,z)\in\Sigma_{2}\times\partial\D\mapsto\left\{\begin{array}{rcl}
	\phantom{\sigma(}(\varepsilon_{1}\varepsilon_{2}\varepsilon_{3}\dots\phantom{)},z^{n_{i,j}})^{\phantom{-}} & \text{if} & \varepsilon_{0}=1 \\
	(\sigma(\varepsilon_{1}\varepsilon_{2}\varepsilon_{3}\dots),z^{-d_{i,j}}) & \text{if} & \varepsilon_{0}=0 \\
\end{array}\right.$$
where
$$\sigma(\varepsilon_{1}\varepsilon_{2}\varepsilon_{3}\dots)=(1-\varepsilon_{1})(1-\varepsilon_{2})(1-\varepsilon_{3})\dots.$$

Notice that every copy of $\partial\D$ which is eventually mapped under iteration by the 2-to-1 map $\sigma$ to the fixed copy coded by $11111\dots$, corresponds to a Jordan curve in $\Julia(f)$ which is eventually mapped under iteration by $f^{p_{i}}$ to $\partial V_{i,j}\subset\partial V_{\infty}$. Therefore, every Julia component which contains such a Jordan curve is actually a preimage of $\partial V_{\infty}$. The others copies correspond to buried Jordan curves, and are either preperiodic or wandering. In the first case (for instance the fixed copy coded by $01010\dots$), the surgery procedure
described in \cite{AutomorphismsRationalMaps} ( Sections 5 and 6)  shows that the Jordan curve is the whole Julia component. The same holds in the second case according to the main result in \cite{RationalMapsDisconnectedJuliaSet}.

Now if $P$ is not affine conjugate to $z\mapsto z^{n}$, every preimage of $\partial V_{\infty}$ comes with infinitely many ``decorations''. In particular, $V_{i,j}$ contains some disjoint closed disks which are preimages by $f^{p_{i}}$ of the whole closed disks $\overline{V_{i,j}}$. Repeating this reasoning gives uncountably many sequences of nested closed disks. It is then straightforward to show (see \cite{SingularPerturbations} or \cite{SingularPerturbationsQuadraticMultiplePoles}) that the intersection of such a nested sequence is actually a point connected component of $\Julia(f)$. The grand orbit of such a point consists of point Julia components as well, and it accumulates everywhere on $\Julia(f)$.
\end{proof}

\begin{thm}\label{ThmJuliaHomeomorphic}
Two McMullen-like mappings are topologically conjugate on their Julia sets (by an orientation preserving homeomorphism) if and only if they have same type $(P,\Pole)$ up to conjugation by an affine map.
\end{thm}

\begin{proof}
Let $f_{1}$ and $f_{2}$ be two McMullen-like mappings, of type $(P_{1},\Pole_{1})$ and $(P_{2},\Pole_{2})$ respectively, which are topologically conjugate on their Julia sets. Since they both have only one Fatou domain which contains the image of every trap door, it follows from the first property of the definition of a McMullen-like mapping (Definition \ref{DefMcMullenLikeMap}) that $P_{1}$ and $P_{2}$ are topologically conjugate on their Julia sets, and hence are affine conjugate as $\HPcFP$. It is straightforward to see that the combinatorial description of all Julia components in the proof of Theorem \ref{ThmJuliaComponents} concludes the proof.
\end{proof}


\subsection{Known examples}

In this section we show some known examples of McMullen-like mappings from the literature. For each example we focus in the pair $(P,\mathcal  D)$ where $P$ is a $\HPcFP$ and  $\Pole$ is the pole data  in the definition of a McMullen-like mapping (see Definition \ref{DefMcMullenLikeMap}).

The first example is the McMullen family given by $f_{\lambda}(z)=z^n+\lambda/z^d$. This rational map has $\infty$ as a super-attracting fixed point  and the only finite preimage of $\infty$ is the origin. Thus if $V_{\infty}$ is the immediate basin of attraction of $\infty$ and when $V_{\infty}$  does not contain the origin, then there exists a unique trap door $T_0$, a neighborhood of the origin,  that is mapped $d$-to-1 onto $V_{\infty}$. We also recall that  $f_{\lambda}$ has $n+d$ ``free'' critical points given by $c_{\lambda}=(\frac{\lambda d}{n})^{1/(n+d)}$ and the corresponding free critical values are $v_{\lambda}=f_{\lambda}(c_{\lambda})$. In Figure \ref{fig:mcmullen_ex1} we show the dynamical plane of the McMullen-like mapping $z^3-0.01/z^3$.

In this case the polynomial $z\mapsto z^n$ is a $\HPcFP$, this is  a singular case since the Fatou set $\Fatou(P)$ only contains  two simply connected components: the immediate basin of attraction of $\infty$ and the immediate basin of attraction of $0$.  Thus there is a unique bounded Fatou domain which coincides with the unit disk and  denoted by $U_{0}$. The pole data $\Pole$ is formed by $U_{0}$ and a positive integer $d=d_{0}\geqslant 1$. From \cite{AutomorphismsRationalMaps} and \cite{SingularPerturbations} we conclude the following result. 

\begin{pro}\label{ex1:McMullen}
Let $f_{\lambda}(z)=z^n + \lambda/z^d$ such that the free critical values $v_{\lambda}$ belong to the trap door $T_0$ and the arithmetic condition $1/n+1/d<1$ is satisfied,  then $f_{\lambda}$ is a McMullen-like mapping.
\end{pro}

The next two examples of McMullen-like mappings are related to the polynomial $P_{c}(z)=z^n+c$ where $c$ is a center of a hyperbolic component of the corresponding Multibrot set. The choice of the parameter $c$ ensures that  $P_{c}$ is $\HPcFP$ since the orbit of the critical point located at $0$ is a periodic orbit of period $p$. We denote by $U_{0},U_{1},\cdots,U_{p-1}$ the Fatou domain containing $0,P_{c}(0),\cdots,P_{c}^{p-1}(0)$, respectively.  

The first McMullen-like mapping related to $P_{c}$ was introduced in \cite{GeneralizedMcMullenDomain} where the authors studied the family of rational maps $g_{\lambda}(z)=z^n+c+\lambda/z^n$. In this case the pole data is formed by the Fatou domain $U_{0}$ and the positive integer $n$, making thus the first generalization of the McMullen family. In Figure \ref{fig:mcmullen_ex1} we show the dynamical plane of the McMullen-like mapping $z^3+i -10^{-7}/z^3$. In \cite{GeneralizedMcMullenDomain} the authors proved the following result.

\begin{pro}\label{ex2:GeneralizedMcMullen}
Let $g_{\lambda}(z)=z^n+c+\lambda/z^n$. For small enough values of  $|\lambda|\neq 0$ and when the arithmetic condition $n>2$ is satisfied, the rational map $g_{\lambda}$ is a McMullen-like mapping.
\end{pro}

The second McMullen-like mapping related to the  $P_{c}(z)=z^n+c$ was introduced in \cite{SingularPerturbationsQuadraticMultiplePoles} where the authors studied the family of rational maps $h_{\lambda}(z)=z^n+c+\lambda/\prod_{j=0}^{p-1}(z-c_{j})^{d_{j}}$, where $c_{j}=P_{c}^{j}(0)$ for $0\leqslant j\leqslant p-1$. In this case  $h_{\lambda}$ has a  pole  in every point of the super-attracting orbit $0\to P_{c}(0)\to\cdots P_{c}^{p-1}(0)$ of order $d_{j}$, for $0\leqslant j\leqslant p-1$.  The pole data $\Pole$ is formed by the Fatou domains $\{U_{0},U_{1},\cdots,U_{p-1}\}$, and the corresponding positive integers $d_{0},d_{1},\cdots,d_{p-1}$. In Figure \ref{fig:mcmullen_ex1} we show the dynamical plane of the McMullen-like mapping $z^2-1+10^{-22}/(z^7(z+1)^5)$. In  \cite{SingularPerturbationsQuadraticMultiplePoles}  it is proved that

\begin{pro}\label{ex3:Multipoles}
Let $h_{\lambda}(z)=z^n+c+\lambda/\prod_{j=0}^{p-1}(z-c_{j})^{d_{j}}$. For small enough values of $|\lambda|\neq 0$ and when the arithmetic conditions $2d_{1}>d_{0}+2$ and $d_{j+1}>d_{j}+1$ for every $1\leqslant j\leqslant p-1$ (with $d_{p}=d_{0}$) are satisfied, the rational map $h_{\lambda}$ is a McMullen-like mapping.
\end{pro}

We notice that in these three examples, one arithmetic condition is required to ensure that the the corresponding rational map is a McMullen-like mapping. In the next section we state the main result of this paper that characterize a McMullen-like mapping of type $(P,\Pole)$ using an arithmetic condition.


\section{Existence of McMullen-like mappings}\label{sec:existence}

\subsection{The arithmetic condition: Theorem A}

\begin{thmA}
Let $\Pole$ be a pole data associated to a $\HPcFP$ $P$. Then there exists a McMullen-like mapping of type $(P,\Pole)$ if and only if the following arithmetic condition holds.
\begin{equation}\label{EqArithmeticCondition}
	\max_{1\leqslant i\leqslant N}\left\{\prod_{\substack{j\in\Z/p_{i}\Z \\ U_{i,j}\notin\Pole}}\frac{1}{n_{i,j}}\times\prod_{\substack{j\in\Z/p_{i}\Z \\ U_{i,j}\in\Pole}}\left(\frac{1}{n_{i,j}}+\frac{1}{d_{i,j}}\right)\right\}<1
	\tag{$\star$}
\end{equation}
\end{thmA}

Let us give some remarks about this arithmetic condition.

\begin{rem}
It is enough to check  the arithmetic condition (\ref{EqArithmeticCondition}) for every $1\leqslant i\leqslant N$ such that the set of indices $J_{i}=\{j\in\Z/p_{i}\Z\ /\ U_{i,j}\in\Pole\}$ is not empty. Indeed, for every $1\leqslant i\leqslant N$ there is at least one degree $n_{i,j}\geqslant 2$, and hence $\prod_{j\in\Z/p_{i}\Z}\frac{1}{n_{i,j}}<1$. Moreover, notice that if $n_{i,j}\geqslant 2$ and $d_{i,j}\geqslant 3$ for every $U_{i,j}\in\Pole$, then the arithmetic condition (\ref{EqArithmeticCondition}) holds. Equivalently speaking, it is sufficient that every periodic bounded Fatou domain picked in the pole data contains a critical point, and every degree in the pole data is larger than $3$.
\end{rem}
 
\begin{rem}
If $P$ is affine conjugate to $z\mapsto z^{n}$, namely if $N=1$ and $p_{1}=1$, then the arithmetic condition (\ref{EqArithmeticCondition}) implies the well known arithmetic condition $\frac{1}{n}+\frac{1}{d}<1$ (see \cite{AutomorphismsRationalMaps} and Proposition \ref{ex1:McMullen}). If $P$ is affine conjugate to $z\mapsto z^{n}+c$ where $c$ is a center of a hyperbolic component of the corresponding Multibrot set, and if $\Pole$ only consists of the bounded Fatou domain containing the unique critical point with $d\geqslant 3$, then there exists a McMullen-like mapping of type $(P,\Pole)$ (see Propositions \ref{ex2:GeneralizedMcMullen} and \ref{pro:new2}). The arithmetic condition (\ref{EqArithmeticCondition}) implies the arithmetic condition introduced in \cite{SingularPerturbationsQuadraticMultiplePoles} (see Proposition \ref{ex3:Multipoles}), but the converse is not true (see Proposition \ref{pro:new3}).
\end{rem}

\begin{rem}\label{rem:degree}
It is straightforward to show that  the degree of a McMullen-like mapping, $\deg(f)=\deg(P)+\deg(\Pole)$,  has a lower bound according to condition (\ref{EqArithmeticCondition}). This lower bound is reached for a polynomial $P$ of degree $n=3$ with two simple critical points in a same super-attracting cycle of period 2, and a pole data which only consists of one of the two bounded Fatou domains containing a critical point with $d=1$ (see Proposition \ref{pro:new1}). Indeed, in that case (\ref{EqArithmeticCondition}) reduces to $\frac{1}{2}(\frac{1}{2}+\frac{1}{1})=\frac{3}{4}<1$.  In particular, there is no McMullen-like mappings of degree less than 4. We notice that the first example of a rational maps of degree less than 5 with buried Julia components was recently founded in \cite{FamilyBuriedJuliaComponents} for an explicit family of cubic rational maps.
\end{rem}

\begin{rem}
Finally, and according to Theorem \ref{ThmJuliaHomeomorphic}, the set of all types $(P,\Pole)$, where $P$ is a monic centered $\HPcFP$ and $\Pole$ is an associated pole data which satisfies condition (\ref{EqArithmeticCondition}), is in 1-to-1 correspondence with the set of all topological conjugation classes of Julia sets of McMullen-like mappings.
\end{rem}


\subsection{Thurston obstruction: Proof of Theorem A}\label{subsec:necessity}

The main ingredient to prove the necessity of the arithmetic condition (\ref{EqArithmeticCondition}) is the theory of Thurston obstructions for rational maps (see \cite{ThurstonProof} and \cite{McMullenBook}).

Assume there exists a McMullen-like mapping $f$ of a given type $(P,\Pole)$. Denote by $\Postcrit_{f}$ the closure of its postcritical set. Fix $1\leqslant i\leqslant N$ such that the set of indices $J_{i}=\{j\in\Z/p_{i}\Z\ /\ U_{i,j}\in\Pole\}$ is not empty. Remark that every $U_{i,j}$ may be uniquely written as $U_{i,j'+k}$ for some $U_{i,j'}\in\Pole$ and some $1\leqslant k\leqslant t_{i,j'}$.

For every $j\in\Z/p_{i}\Z$, consider an arbitrary Jordan curve $\Gamma_{i,j}$ in $V_{i,j}$ such that
\begin{itemize}
\item $\Gamma_{i,j}$ separates $\partial V_{i,j}$ from $\overline{f^{k}(A_{i,j'})}\subset V_{i,j}$ if $U_{i,j}=U_{i,j'+k}\notin\Pole$ with $1\leqslant k<t_{i,j'}$;
\item $\Gamma_{i,j}$ separates $\partial V_{i,j}$ from $\overline{f^{t_{i,j'}}(A_{i,j'})}=\overline{T_{i,j}}$ (by Lemma \ref{LemPreimageTrapDoor}) if $U_{i,j}=U_{i,j'+t_{i,j'}}\in\Pole$.
\end{itemize}
From Definition \ref{DefMcMullenLikeMap} and Lemma \ref{LemDynamicsOnVij}, every $f^{k}(A_{i,j'})$ contains at least $n_{i,j'}+d_{i,j'}\geqslant 2$ postcritical points. Thus, it is straightforward to show that $\Gamma_{i}=\{\Gamma_{i,j}\ /\ j\in\Z/p_{i}\Z\}$ is a multicurve, namely a finite collection of disjoint, non-homotopic, and non-peripheral Jordan curves in $\CC\setminus\Postcrit_{f}$ (recall that a Jordan curve in $\CC\setminus\Postcrit_{f}$ is said to be non-peripheral if each connected component of its complement contains at least two points in $\Postcrit_{f}$).

For every $\Gamma_{i,j}\in\Gamma_{i}$, consider the connected components of $f^{-1}(\Gamma_{i,j})$ which are homotopic in $\CC\setminus\Postcrit_{f}$ to some Jordan curves in $\Gamma_{i}$. According to Definition \ref{DefMcMullenLikeMap} and Lemma \ref{LemDynamicsOnVij}, any such connected component lies in $U_{i,j-1}$, and hence is homotopic in $\CC\setminus\Postcrit_{f}$ to $\Gamma_{i,j-1}$. More precisely, there are:
\begin{itemize}
	\item only one connected component if $U_{i,j-1}\notin\Pole$, which is mapped onto $\Gamma_{i,j}$ with degree $n_{i,j-1}$;
	\item two connected components if $U_{i,j-1}\in\Pole$, one in $A_{i,j-1}^{\outer}$ and one in $A_{i,j-1}^{\inner}$, which are mapped onto $\Gamma_{i,j}$ with degrees $n_{i,j-1}$ and $d_{i,j-1}$ respectively.
\end{itemize}
It follows that the transition matrix associated to the multicurve $\Gamma_{i}$ may be written as
$$\left(\vcenter{\xymatrix @R=1pt @C=1pt {
	0 \ar@{-}[rrrrrrdddddd] & m_{1,2} & 0 \ar@{-}[rrrr] \ar@{-}[rrrrdddd] &&&& 0 \ar@{-}[dddd] \\
	0 \ar@{-}[dddd] \ar@{-}[rrrrrddddd] && m_{2,3} \ar@{-}[rrrrdddd] &&&& \\
	&&&&&& \\
	&&&&&& \\
	&&&&&& 0 \\
	0 \ar@{-}[rd] &&&&&& m_{p_{i}-1,p_{i}} \\ 
	m_{p_{i},1} & 0 \ar@{-}[rrrr] &&&& 0 & 0 \\
}}\right)\quad\text{where}\quad m_{j-1,j}=\left\{\begin{array}{ccl}
	\frac{1}{n_{i,j-1}} & \text{if} & U_{i,j-1}\notin\Pole \\
	\frac{1}{n_{i,j-1}}+\frac{1}{d_{i,j-1}} & \text{if} & U_{i,j-1}\in\Pole \\
\end{array}\right..$$
It is straightforward to show that the associated leading eigenvalue is
$$\lambda(\Gamma_{i})=\left[\prod_{j\in\Z/p_{i}\Z}m_{j-1,j}\right]^{\frac{1}{p_{i}}}=\left[\prod_{\substack{j\in\Z/p_{i}\Z \\ U_{i,j}\notin\Pole}}\frac{1}{n_{i,j}}\times\prod_{\substack{j\in\Z/p_{i}\Z \\ U_{i,j}\in\Pole}}\left(\frac{1}{n_{i,j}}+\frac{1}{d_{i,j}}\right)\right]^{\frac{1}{p_{i}}}.$$

Now applying Theorem B.3 and Theorem B.4 from \cite{McMullenBook} to the hyperbolic rational map $f$, we get $\lambda(\Gamma_{i})<1$ for every $1\leqslant i\leqslant N$ such that $J_{i}\neq\emptyset$, that implies the arithmetic condition (\ref{EqArithmeticCondition}).


\subsection{Construction of McMullen-like mappings: Proof of Theorem A}\label{subsec:construction}

In this section, we construct a McMullen-like mapping of an arbitrary given type $(P,\Pole)$ which satisfies the arithmetic condition (\ref{EqArithmeticCondition}). The method is to start from the polynomial $P$ and to add a pole of degree $d_{i,j}$ on every $U_{i,j}\in\Pole$ by quasiconformal surgery (we refer readers to \cite{QuasiconformalSurgeryBook} for a comprehensive treatment on this powerful method). At first, we need the two following technical lemmas that will allow us to divide the Riemann sphere $\CC$ into several pieces on which a quasiregular map will be defined.

\begin{lem}\label{LemTechnic}
If the arithmetic condition (\ref{EqArithmeticCondition}) holds then there exist some positive real numbers $(\alpha_{i,j})_{U_{i,j}\in\Pole}$ and $(\beta_{i,j})_{U_{i,j}\in\Pole}$ such that
$$\forall U_{i,j}\in\Pole,\quad\alpha_{i,j}<\beta_{i,j}\quad\text{and}\quad\left(\frac{n_{i,j+t_{i,j}}}{d_{i,j+t_{i,j}}}\alpha_{i,j+t_{i,j}}+\beta_{i,j+t_{i,j}}\right)<\left(\prod_{k=0}^{t_{i,j}-1}n_{i,j+k}\right)\alpha_{i,j}.$$
\end{lem}

\begin{proof}
Fix $1\leqslant i\leqslant N$ such that the set of indices $J_{i}=\{j\in\Z/p_{i}\Z\ /\ U_{i,j}\in\Pole\}$ is not empty, and let $M$ be the following positive constant.
$$M=\left[\prod_{\substack{j\in\Z/p_{i}\Z \\ U_{i,j}\notin\Pole}}\frac{1}{n_{i,j}}\times\prod_{\substack{j\in\Z/p_{i}\Z \\ U_{i,j}\in\Pole}}\left(\frac{1}{n_{i,j}}+\frac{1}{d_{i,j}}\right)\right]^{-1/2|J_{i}|}=\prod_{j\in J_{i}}\left[\prod_{k=1}^{t_{i,j}-1}\frac{1}{n_{i,j+k}}\times\left(\frac{1}{n_{i,j}}+\frac{1}{d_{i,j}}\right)\right]^{-1/2|J_{i}|}$$
Remark that the arithmetic condition (\ref{EqArithmeticCondition}) precisely implies that $M>1$.

Now pick one $j_{0}\in J_{i}$, let $\alpha_{i,j_{0}}$ be any positive real number, and recursively define $\alpha_{i,j}$ for every $j\in J_{i}$ by
\begin{equation}\label{EqLemTechnic}
	\alpha_{i,j+t_{i,j}}=\left(M^{-2}\frac{n_{i,j}}{n_{i,j+t_{i,j}}}\left[\prod_{k=1}^{t_{i,j}-1}\frac{1}{n_{i,j+k}}\times\left(\frac{1}{n_{i,j+t_{i,j}}}+\frac{1}{d_{i,j+t_{i,j}}}\right)\right]^{-1}\right)\alpha_{i,j}.
\end{equation}
The $(\alpha_{i,j})_{j\in J_{i}}$ are well defined because the product over $j\in J_{i}$ of the terms in the biggest brackets is $1$ by definition of $M$. Define $\beta_{i,j}$ for every $j\in J_{i}$ so that
$$Mn_{i,j}\left(\frac{1}{n_{i,j}}+\frac{1}{d_{i,j}}\right)\alpha_{i,j}=\frac{n_{i,j}}{d_{i,j}}\alpha_{i,j}+\beta_{i,j}\quad\text{or equivalently}\quad\beta_{i,j}=M\alpha_{i,j}+(M-1)\frac{n_{i,j}}{d_{i,j}}\alpha_{i,j}.$$
Remark that $\beta_{i,j}>\alpha_{i,j}$ since $M>1$. Moreover, (\ref{EqLemTechnic}) may be rewritten as follows
$$Mn_{i,j+t_{i,j}}\left(\frac{1}{n_{i,j+t_{i,j}}}+\frac{1}{d_{i,j+t_{i,j}}}\right)\alpha_{i,j+t_{i,j}}=M^{-1}\left(\prod_{k=0}^{t_{i,j}-1}n_{i,j+k}\right)\alpha_{i,j}.$$
Using the definition of $\beta_{i,j+t_{i,j}}$ above concludes the proof since $M^{-1}<1$.
\end{proof}

For every periodic bounded Fatou domain $U_{i,j}$ (not necessarily in $\Pole$), Böttcher Theorem provides a Riemann mapping $\phi_{i,j}:\D\rightarrow U_{i,j}$ such that the following diagram commutes.

$$\xymatrix{
	\D \ar[rr]^{\textstyle z\mapsto z^{n_{i,j}}} \ar[d]_{\textstyle \phi_{i,j}} && \D \ar[rr]^{\textstyle z\mapsto z^{n_{i,j+1}}} \ar[d]_{\textstyle \phi_{i,j+1}} && \D \ar[r] \ar[d]_{\textstyle \phi_{i,j+2}} & \dots \ar[r] & \D \ar[rr]^{\textstyle z\mapsto z^{n_{i,j+p_{i}-1}}} \ar[d]_{\textstyle \phi_{i,j+p_{i}-1}} && \D \ar[d]_{\textstyle \phi_{i,j+p_{i}}=\phi_{i,j}} \\
	U_{i,j} \ar[rr]_{\textstyle P} && U_{i,j+1} \ar[rr]_{\textstyle P} && U_{i,j+2} \ar[r] & \dots \ar[r] & U_{i,j+p_{i}-1} \ar[rr]_{\textstyle P} && U_{i,j+p_{i}}=U_{i,j} \\
}$$

An equipotential $\gamma$ in some $U_{i,j}$ is the image by $\phi_{i,j}$ of an Euclidean circle in $\D$ centered at $0$. The radius of this circle is called the level of $\gamma$, and is denoted by $L_{i,j}(\gamma)\in]0,1[$ in order that $\gamma=\{z\in U_{i,j}\ /\ |\phi_{i,j}^{-1}(z)|=L_{i,j}(\gamma)\}$.

Similarly, Böttcher Theorem provides a Riemann mapping $\phi_{\infty}:\CC\setminus\D\rightarrow U_{\infty}$ which conjugates $P|_{U_{\infty}}$ with $z\mapsto z^{n}$. The equipotentials in $U_{\infty}$ are defined as well (with level $>1$).

Recall that any pair of disjoint continua $\gamma,\gamma'$ in $\CC$ uniquely defines a doubly connected domain denoted by $\annulus{\gamma}{\gamma'}$. If $\gamma,\gamma'$ contain at least two points each, $\annulus{\gamma}{\gamma'}$ is biholomorphic to a round annulus of the form $\{z\in\C\ /\ r<|z|<1\}$ where $r\in]0,1[$ does not depend on the choice of biholomorphism. The modulus of $\annulus{\gamma}{\gamma'}$ is defined to be $\modulus{\gamma}{\gamma'}=\frac{1}{2\pi}\log(\frac{1}{r})$. In particular if $\gamma,\gamma'$ are two equipotentials in some $U_{i,j}$ of levels $L_{i,j}(\gamma)>L_{i,j}(\gamma')$ then
$$\modulus{\gamma}{\gamma'}=\frac{1}{2\pi}\log\left(\frac{L_{i,j}(\gamma)}{L_{i,j}(\gamma')}\right).$$

\begin{lem}\label{LemEquipotentials}
If the arithmetic condition (\ref{EqArithmeticCondition}) holds then there exist an equipotential $\Gamma_{\infty}$ in $U_{\infty}$, and three equipotentials $\gamma_{i,j}^{\outer},\gamma_{i,j}^{\inner},\gamma_{i,j}^{\infty}$ in every $U_{i,j}\in\Pole$ such that
\begin{description}
	\item[(i)] $L_{i,j}(\gamma_{i,j}^{\outer})>L_{i,j}(\gamma_{i,j}^{\inner})>L_{i,j}(\gamma_{i,j}^{\infty})$;
	\item[(ii)] $\modulus{\gamma_{i,j}^{\inner}}{\gamma_{i,j}^{\infty}}=\frac{1}{d_{i,j}}\modulus{\Gamma_{i,j+1}}{\Gamma_{\infty}}$ where $\Gamma_{i,j+1}=P(\gamma_{i,j}^{\outer})$;
	\item[(iii)] $L_{i,j+t_{i,j}}(\gamma_{i,j+t_{i,j}}^{\infty})>\left(\prod_{k=0}^{t_{i,j}-1}n_{i,j+k}\right)L_{i,j}(\gamma_{i,j}^{\outer})$.
\end{description}
\end{lem}

Notice that $\Gamma_{i,j+1}$ is actually an equipotential in $U_{i,j+1}$ of level $L_{i,j+1}(\Gamma_{i,j+1})=n_{i,j}L_{i,j}(\gamma_{i,j}^{\outer})$.

\begin{proof}
Fix $\Gamma_{\infty}$ to be any arbitrary equipotential in $U_{\infty}$. Let $r$ be a real number in $]0,1[$. In each $U_{i,j}\in\Pole$, define $\gamma_{i,j}^{\outer}$ and $\gamma_{i,j}^{\inner}$ to be the equipotentials of levels $L_{i,j}(\gamma_{i,j}^{\outer})=r^{\alpha_{i,j}}$ and $L_{i,j}(\gamma_{i,j}^{\inner})=r^{\beta_{i,j}}$ respectively, where $\alpha_{i,j}$ and $\beta_{i,j}$ come from Lemma \ref{LemTechnic}. Notice that $L_{i,j}(\gamma_{i,j}^{\outer})>L_{i,j}(\gamma_{i,j}^{\inner})$ since $\alpha_{i,j}<\beta_{i,j}$. Define $\gamma_{i,j}^{\infty}$ in every $U_{i,j}\in\Pole$ so that the point \textbf{(ii)} holds, or equivalently $L_{i,j}(\gamma_{i,j}^{\infty})=r^{\delta_{i,j}}$ with
$$\delta_{i,j}=\beta_{i,j}+\dfrac{\frac{2\pi}{d_{i,j}}\modulus{\Gamma_{i,j+1}}{\Gamma_{\infty}}}{\ln\left(\frac{1}{r}\right)}.$$
In particular, $L_{i,j}(\gamma_{i,j}^{\inner})>L_{i,j}(\gamma_{i,j}^{\infty})$ since $\beta_{i,j}<\delta_{i,j}$ that completes the point \textbf{(i)}. For the last point, we have by using the inverse Grötzch inequality due to Cui Guizhen and Tan Lei (see Appendix B in \cite{ThurstonSubHyperbolic}):
$$\modulus{\Gamma_{i,j+t_{i,j}+1}}{\Gamma_{\infty}}\leqslant C+\frac{1}{2\pi}\ln\left(\frac{1}{L_{i,j+t_{i,j}+1}(\Gamma_{i,j+t_{i,j}+1})}\right)=C+\frac{n_{i,j+t_{i,j}}}{2\pi}\alpha_{i,j+t_{i,j}}\ln\left(\frac{1}{r}\right)$$
where $C>0$ does not depend on $r$. Putting this in the expression of $\delta_{i,j+t_{i,j}}$ leads to
$$\delta_{i,j+t_{i,j}}\leqslant\left(\frac{n_{i,j+t_{i,j}}}{d_{i,j+t_{i,j}}}\alpha_{i,j+t_{i,j}}+\beta_{i,j+t_{i,j}}\right)+\frac{2\pi C}{d_{i,j+t_{i,j}}\ln\left(\frac{1}{r}\right)}<\left(\prod_{k=0}^{t_{i,j}-1}n_{i,j+k}\right)\alpha_{i,j}$$
provided $r>0$ is small enough according to Lemma \ref{LemTechnic}. The point \textbf{(iii)} follows.
\end{proof}

Now we are going to piecewisely define a quasiregular map $F$ on $\CC$ according to a partition induced by the equipotentials coming from Lemma \ref{LemEquipotentials}. Let $D_{\infty}$ be the unbounded connected component of $\CC\setminus\Gamma_{\infty}$, and $W_{\infty}$ be the unbounded connected component of $\CC\setminus\bigcup_{U_{i,j}\in\Pole}\gamma_{i,j}^{\outer}$. Denote by $D(\gamma)$ the bounded connected component of $\CC\setminus\gamma$ for every Jordan curve $\gamma$ in $\C$. Consider the following partition of $\CC$:
$$\CC=W_{\infty}\bigcup_{U_{i,j}\in\Pole}\left(\overline{A(\gamma_{i,j}^{\outer},\gamma_{i,j}^{\inner})}\cup A(\gamma_{i,j}^{\inner},\gamma_{i,j}^{\infty})\cup\overline{D(\gamma_{i,j}^{\infty}})\right).$$

\begin{itemize}
\item On $W_{\infty}$:\\
	Define $F$ to be the polynomial $P$ in order that $F$ continuously extend to every $\gamma_{i,j}^{\outer}$ with $F(\gamma_{i,j}^{\outer})=P(\gamma_{i,j}^{\outer})=\Gamma_{i,j+1}$ and $\deg(F|_{\gamma_{i,j}^{\outer}})=n_{i,j}$.
\item On every $A(\gamma_{i,j}^{\inner},\gamma_{i,j}^{\infty})$:\\
	Define $F$ to be such that $F(A(\gamma_{i,j}^{\inner},\gamma_{i,j}^{\infty}))=A(\Gamma_{i,j+1},\Gamma_{\infty})$ and $F|_{A(\gamma_{i,j}^{\inner},\gamma_{i,j}^{\infty})}$ is a holomorphic covering of degree $d_{i,j}$ which continuously extends on the boundary with $F(\gamma_{i,j}^{\inner})=\Gamma_{i,j+1}$ and $F(\gamma_{i,j}^{\infty})=\Gamma_{\infty}$. The point \textbf{(ii)} in Lemma \ref{LemEquipotentials} ensures that such a holomorphic covering exists. Notice that $\deg(F|_{\gamma_{i,j}^{\inner}})=\deg(F|_{\gamma_{i,j}^{\infty}})=d_{i,j}$.
\item On every $\overline{A(\gamma_{i,j}^{\outer},\gamma_{i,j}^{\inner})}$:\\
	Continuously extend $F$ so that $F(A(\gamma_{i,j}^{\outer},\gamma_{i,j}^{\inner}))=D(\Gamma_{i,j+1})$ and $F|_{A(\gamma_{i,j}^{\outer},\gamma_{i,j}^{\inner})}$ is a quasiregular branched covering. This extension must have degree $\deg(F|_{\gamma_{i,j}^{\outer}})+\deg(F|_{\gamma_{i,j}^{\inner}})=n_{i,j}+d_{i,j}$, and $n_{i,j}+d_{i,j}$ critical points counted with multiplicity by the Riemann-Hurwitz formula (compare with the action of a McMullen-like mapping on every $A_{i,j}$ in Lemma \ref{LemDynamicsOnVij}). The existence of such an annulus-disk map is discussed in \cite{DiskAnnulusSurgery} (see also \cite{QuasiconformalSurgeryBook}) and \cite{FamilyBuriedJuliaComponents}. 
\item On every $\overline{D(\gamma_{i,j}^{\infty})}$:\\
	Continuously extend $F$ so that $F(D(\gamma_{i,j}^{\infty}))=D_{\infty}$ and $F|_{D(\gamma_{i,j}^{\infty})}$ is a quasiregular branched covering. This extension must have degree $\deg(F|_{\gamma_{i,j}^{\infty}})=d_{i,j}$, and $d_{i,j}-1$ critical points counted with multiplicity by the Riemann-Hurwitz formula (compare with the action of a McMullen-like mapping on every $T_{i,j}$ in Lemma \ref{LemDynamicsOnVij}). To construct such a quasiregular branched covering, we may start from the map
$$\phi_{\infty} \circ\left(z\mapsto L_{\infty}(\Gamma_{\infty})\left(\frac{L_{i,j}(\gamma_{i,j}^{\infty})}{z}\right)^{d_{i,j}}\right)\circ\phi_{i,j}^{-1}$$
which is a holomorphic branched covering of degree $d_{i,j}$ from $D(\gamma_{i,j}^{\infty})$ onto $D_{\infty}$, and then quasiconformally modify it in a neighborhood of $\gamma_{i,j}^{\infty}$ in order that the continuous extension on $\gamma_{i,j}^{\infty}$ agrees with the definition of $F$ on $A(\gamma_{i,j}^{\inner},\gamma_{i,j}^{\infty})$.
\end{itemize}

Remark that $F$ is actually holomorphic outside $\bigcup_{U_{i,j}\in\Pole}\left(\overline{A(\gamma_{i,j}^{\outer},\gamma_{i,j}^{\inner})}\cup\overline{D(\gamma_{i,j}^{\infty}})\right)$. The following lemma shows that this set does not intersect its forward orbit after finitely many iterations.

\begin{lem}\label{LemNoQuasiconformalRecurrence}
For every $U_{i,j}\in\Pole$, the following hold
$$F^{t_{i,j}}\left(\overline{A(\gamma_{i,j}^{\outer},\gamma_{i,j}^{\inner})}\right)\subset D(\gamma_{i,j+t_{i,j}}^{\infty}),\quad F\left(\overline{D(\gamma_{i,j}^{\infty})}\right)=\overline{D_{\infty}},\quad\text{and}\quad F\left(\overline{D_{\infty}}\right)\subset D_{\infty}\subset W_{\infty}.$$

\begin{proof}
By definition of $F$ on $\overline{A(\gamma_{i,j}^{\outer},\gamma_{i,j}^{\inner})}$ and $W_{\infty}$, we have
$$F^{t_{i,j}}\left(\overline{A(\gamma_{i,j}^{\outer},\gamma_{i,j}^{\inner})}\right)=F^{t_{i,j}-1}\left(\overline{D(\Gamma_{i,j+1})}\right)=P^{t_{i,j}-1}\left(\overline{D(\Gamma_{i,j+1})}\right)=\overline{D(\Gamma_{i,j+t_{i,j}})}$$
where $\Gamma_{i,j+t_{i,j}}=P^{t_{i,j}-1}(\Gamma_{i,j+1})=P^{t_{i,j}}(\gamma_{i,j}^{\outer})$ is an equipotential in $U_{i,j+t_{i,j}}$ of level
$$L_{i,j+t_{i,j}}(\Gamma_{i,j+t_{i,j}})=\left(\prod_{k=0}^{t_{i,j}-1}n_{i,j+k}\right)L_{i,j}(\gamma_{i,j}^{\outer}).$$
The point \textbf{(iii)} in Lemma \ref{LemEquipotentials} gives $L_{i,j+t_{i,j}}(\gamma_{i,j+t_{i,j}}^{\infty})>L_{i,j+t_{i,j}}(\Gamma_{i,j+t_{i,j}})$ that implies the first inclusion. The equality follows from definition of $F$ on $\overline{D(\gamma_{i,j}^{\infty})}$. Finally, $D_{\infty}\subset U_{\infty}\subset W_{\infty}$ by definition of $W_{\infty}$, and hence $F|_{\overline{D_{\infty}}}=P|_{\overline{D_{\infty}}}$ is conjugate to the action of $z\mapsto z^{n}$ on $\CC\setminus L_{\infty}(\Gamma_{\infty})\D$ that concludes the proof.
\end{proof}
\end{lem}

As a consequence, the iterated pullback by $F$ of the standard complex structure provides a $F$-invariant Beltrami form on $\CC$ with uniformly bounded dilatation. Then, after integrating, there exists a quasiconformal map $\varphi:\CC\rightarrow\CC$ such that $f=\varphi\circ F\circ\varphi^{-1}$ is holomorphic on $\CC$, namely a rational map. We refer readers to \cite{QuasiconformalSurgery} and \cite{QuasiconformalSurgeryBook} for more details about this result known as Shishikura principle for quasiconformal surgery.

Finally, it is straightforward to see that $f$ is actually a McMullen-like mapping of type $(P,\Pole)$ by construction.


\section{Further examples of McMullen-like mappings}\label{sec:new_examples}

In this section we show several new  examples of  McMullen-like mappings. 

In the first example we give an explicit expression of a McMullen-like mapping of minimal degree 4 (see Remark \ref{rem:degree} below Theorem A). We notice that the first example of a rational map of degree less than 5 with buried Julia components was recently founded in  \cite{FamilyBuriedJuliaComponents}. We prove the following result.

\begin{pro}\label{pro:new1}
Let $q_{\lambda}(z)=2z^3-3z^2+1+\lambda/z$. For small enough values of $|\lambda|\neq0$, the rational map $q_{\lambda}$ is a McMullen-like mapping of minimal degree 4.
\end{pro}

\begin{proof}
We first consider the basic dynamics of  the cubic polynomial $Q(z)=2z^3-3z^2+1$. The polynomial $Q$ is a  hyperbolic postcritically finite polynomial.  The critical points of $Q$ are  $\infty$, $0$, and $1$ since $Q'(z)=6z(z-1)$ and  $Q$ has two super-attracting  cycles, the first one is $\infty\mapsto\infty$ and the second one is $0\mapsto  1\mapsto 0$. We denote by $U_{\infty}$, $U_{0}$, and $U_{1}$ the Fatou domains
of $\Fatou(Q)$ containing $\infty$, $0$, and $1$, respectively. The dynamics of $Q$ on these Fatou domains is  $Q:U_{\infty}\to U_{\infty}$ with degree 3,  $Q:U_{0}\to U_{1}$ with degree 2, and $Q:U_{1}\to U_{0}$ also degree 2. In Figure \ref{fig:new_mcmullen} we show the dynamical planes of $Q$ and $q_{\lambda}$.

We show that $q_{\lambda}$ is a McMullen-like mapping of type $(Q,\Pole)$ checking in turn the conditions of Definition \ref{DefMcMullenLikeMap}, where $Q(z)=2z^3-3z^2+1$ and the pole data $\Pole$ is formed by the Fatou domain $U_{0}$ (containing the origin) associated to the integer $d=d_{0}=1$. We prove the first two conditions  in the definition of a McMullen-like mapping (Definition \ref{DefMcMullenLikeMap}) using an holomorphic motion, for $\lambda$ small,  of $U_{\infty}$ parametrized by $\lambda$ obtaining the immediate basin of attraction of $\infty$ of $q_{\lambda}$ as a result of this movement. Applying the $\Lambda-$Lemma, established by Ma\~{n}\'e, Sad and Sullivan (\cite{ManeSadSullivan}), we extend this holomorphic motion to the boundary of $U_{\infty}$ which coincides with $\Julia(Q)$. This establish that the boundary of the immediate basin of attraction of $\infty$ of $q_{\lambda}$ is a holomorphic motion of $\Julia(Q)$ and this new holomorphic motion is precisely the conjugacy between $q_{0}=Q$ acting on $\Julia(Q)$ and $q_{\lambda}$ acting on the boundary of the immediate basin of attraction of $\infty$ of $q_{\lambda}$.  Finally, we prove the last two conditions of Definition \ref{DefMcMullenLikeMap} studying the behavior of the critical points of $q_{\lambda}$.

First of all we can compute the critical points and the critical values of $q_{\lambda}$. Obviously  $z=\infty$ is a superattracting fixed point of $q_{\lambda}$, near infinity the rational map $q_{\lambda}$ is  conformally conjugate to $z\mapsto z^3$. Moreover, since the degree of $q_{\lambda}$ is 4 we have that $q_{\lambda}$ has 6 critical points, $\infty$ with multiplicity 2 and the other 4 critical points are the solution of $q_{\lambda}'(z)=0$. Easy computations show that finite critical points of $q_{\lambda}$ are the solutions of 
\[
6z^3(z-1)=\lambda \, ,
\]
\noindent for small values of $\lambda$ the above equation has three solutions near 0, denoted by $c_{0}(\lambda)$, and one solution near 1, denoted by $c_{1}(\lambda)$.  We have that
\[
c_{0}(\lambda) \simeq \sqrt[3]{\frac{-\lambda}{6}}  \quad \hbox{and} \quad c_{1}(\lambda)\simeq 1+\frac{\lambda}{6}.
\]
The corresponding critical values $v_{0}(\lambda)=q_{\lambda}(c_{0}(\lambda))$ and $v_{1}(\lambda)=q_{\lambda}(c_{1}(\lambda))$, are given by
\[
v_{0}(\lambda) \simeq 1-\frac{\lambda}{3} - \frac{9}{\sqrt[3]{36}} \lambda^{2/3}  \quad \hbox{and} \quad v_{1}(\lambda) \simeq \frac{\lambda}{1+\lambda/6} + \lambda^2/6+ \lambda^3/18 = 
\lambda + \lambda^2/3 + \mathcal O  (\lambda^3).
\]
Thus, we have that $v_{0}(\lambda) \to 1 $ and $v_1{(}\lambda) \to 0$ as $\lambda \to 0$. 

We denote by $V_{\infty}(\lambda)$ the immediate basin of attraction of $z=\infty$ of $q_{\lambda}$. As it is well known, there is a B\"ottcher coordinate $\phi_{\lambda}$ defined in a neighborhood of $\infty$ in $V_{\infty}(\lambda)$ that conjugates $q_{\lambda}$ to $z \to z^3$ in a neighborhood of $\infty$. If none of the free critical points lie in $V_{\infty}(\lambda)$, then it is well known that we may extend $\phi_{\lambda}$ so that it takes the entire immediate basin univalently onto $\C \setminus \overline{\D}$. If $\lambda$ is sufficiently small then none of the free critical points lie in $V_{\infty}(\lambda)$, since they are close to $0$ and $1$.

Thus, we define $\delta>0$ such that the B\"ottcher map $\phi_{\lambda}$ extends to the whole immediate basin of $V_{\infty}(\lambda)$ for all $|\lambda|<\delta$. We denote by $\D_{\delta} = \{ \lambda \in \C, \, |\lambda|<\delta \}$. We consider the following map
 
 \begin{equation*}
\begin{array}{llll}
H:&   V_{\infty}(0) \times \D_{\delta}  & \to &
 \CC  \\
& (z,\lambda) & \to &   \phi^{-1}_{ \lambda} \circ \phi_{0}(z).
  \end{array}
\end{equation*}

Next we verify that $H$ is a holomorphic motion. By construction, we
have that  $H(z,0)=\phi^{-1}_{0} \circ \phi_{0}(z)= z$. If we
fix the parameter $\lambda$ we can see that the map $H(\cdot,\lambda)$ is
injective. This is immediate since the B\"ottcher mapping
$\phi_{\lambda}$ is conformal. Finally, if we fix a point $z \in V_{\lambda}(0)$
we can see that $H(z,\cdot):\D \to \C$ is a holomorphic map. In
this case this map is a composition of holomorphic maps, since the
B\"ottcher map depends analytically on parameters. 

Applying the $\Lambda$-Lemma to $H$, we obtain a new
holomorphic motion $\overline{H}: \overline{V_{\infty}(0)} \times \D_{\delta} \to \CC$.
Hence, it follows that the boundary of $V_{ \infty}(\lambda)$ is the
continuous image under $\overline{H}$ of the Julia set of $Q$ and the following diagram commutes
$$\xymatrix{
	\Julia(Q) \ar[rr]^{\textstyle Q} \ar[d]_{\textstyle \overline{H}} && \Julia(Q)   \ar[d]_{\textstyle \overline{H}}  \\
	\partial V_{\infty}(\lambda) \ar[rr]_{\textstyle q_{\lambda}} && \partial V_{\infty}(\lambda) \\
}$$

\noindent proving thus conditions (i) and (ii) in Definition \ref{DefMcMullenLikeMap}.

Hereafter we consider $|\lambda|<\delta$. We observe that $z=0$ is the unique finite preimage of $z=\infty$, so in this case there exists a unique trap door $T_{0}(\lambda)$ containing the origin and such that $q_{\lambda}:T_{0}(\lambda) \to V_{\infty}(\lambda)$ has degree 1, thus $q_{\lambda}$ has a unique  trap door $T_{0}(\lambda)$ which is a simply connected domain. We can compute now the preimages of $T_{0}(\lambda)$. We claim that the preimages of $T_{0}(\lambda)$ is formed by three simply connected regions $W_{1}(\lambda), W_{-1/2}(\lambda)$ and $W_{0}(\lambda)$. To see the claim we observe that the polynomial  $Q(z)=2z^3-3z^2+1=2(z-1)^2(z+1/2)$ has  three roots, counted with multiplicity, and  are located at $z=1$ a double root and at $z=-1/2$ a simple root.  By continuity a neighborhood of $z=-1/2$, denoted by $W_{-1/2}(\lambda)$  is still mapped under $q_{\lambda}$  to $T_{0}(\lambda)$ with degree 1 and a neighborhood of $z=1$, denoted by $W_{1}(\lambda)$  is still mapped 
under $q_{\lambda}$ to $T_{0}(\lambda)$ with degree 2. Notice that $W_1(\lambda)$ contains the critical point $c_{1}(\lambda)$. Near the origin there exists another preimage of the trap door, denoted by $W_0(\lambda)$ since  $q_{\lambda}$ sends the trap door $T_{0}(\lambda)$ onto $V_{\infty}(\lambda)$.  

We can show now that the three remaining  critical points, denoted by $c_0(\lambda)$, belong to a doubly connected domain $A(\lambda)$ separating $T_{0}(\lambda)$ and $V_{0}=V(U_{0})$. We can compute the preimage of $W_{1}(\lambda)$ by $q_{\lambda}$. First we observe that $Q(3/2)=1$ since $Q(z)=z^2(2z-3)+1$, so again by continuity there exists a simply connected domain near $z=3/2$ that is mapped under $q_{\lambda}$ onto $W_1(\lambda)$. As we show before the corresponding critical values $v_{0}(\lambda)=q_{\lambda}(c_{0}(\lambda))$ are close to 1, for $\lambda$ sufficiently small they are in $W_{1}(\lambda)$. Thus the three critical points  $c_{0}(\lambda)$ belong to a preimage of $W_{1}(\lambda)$ and  applying the Riemann-Hurwitz formula we obtain that the three critical points belong to a double connected domain $A(\lambda)$, obtaining thus that $q_{\lambda}:A(\lambda)\to W_1(\lambda)$ with degree 3 and since $W_{1}(\lambda)$ is a preimage of the trap door we have that $q^2_{\lambda}(A(\lambda))=q_{\lambda}(W_{1}(\lambda))=T_{0}(\lambda)$. Since $q_{\lambda}$ has no other critical points than $c_{0}(\lambda)$ and $c_{1}(\lambda)$ this prove that for $\lambda$ sufficiently small $q_{\lambda}$ is a McMullen-like mapping.

Finally, we can check the arithmetic condition  (\ref{EqArithmeticCondition}) for the McMullen-like mapping $q_{\lambda}$. We have that $U_{0}\in\Pole$ and  $Q:U_{0}\to U_{1}$ with degree 2 and 
$U_{1}\notin\Pole$ and $Q:U_{1}\to U_{0}$ with degree 2, and $d=d_{0}=1$, so the arithmetic condition writes as
\[
\frac{1}{2} \left(\frac{1}{2}+\frac{1}{1} \right) = \frac{3}{4} < 1. 
\]
\end{proof}

\begin{figure}[ht]
	\centering
	\includegraphics[width=0.8\textwidth]{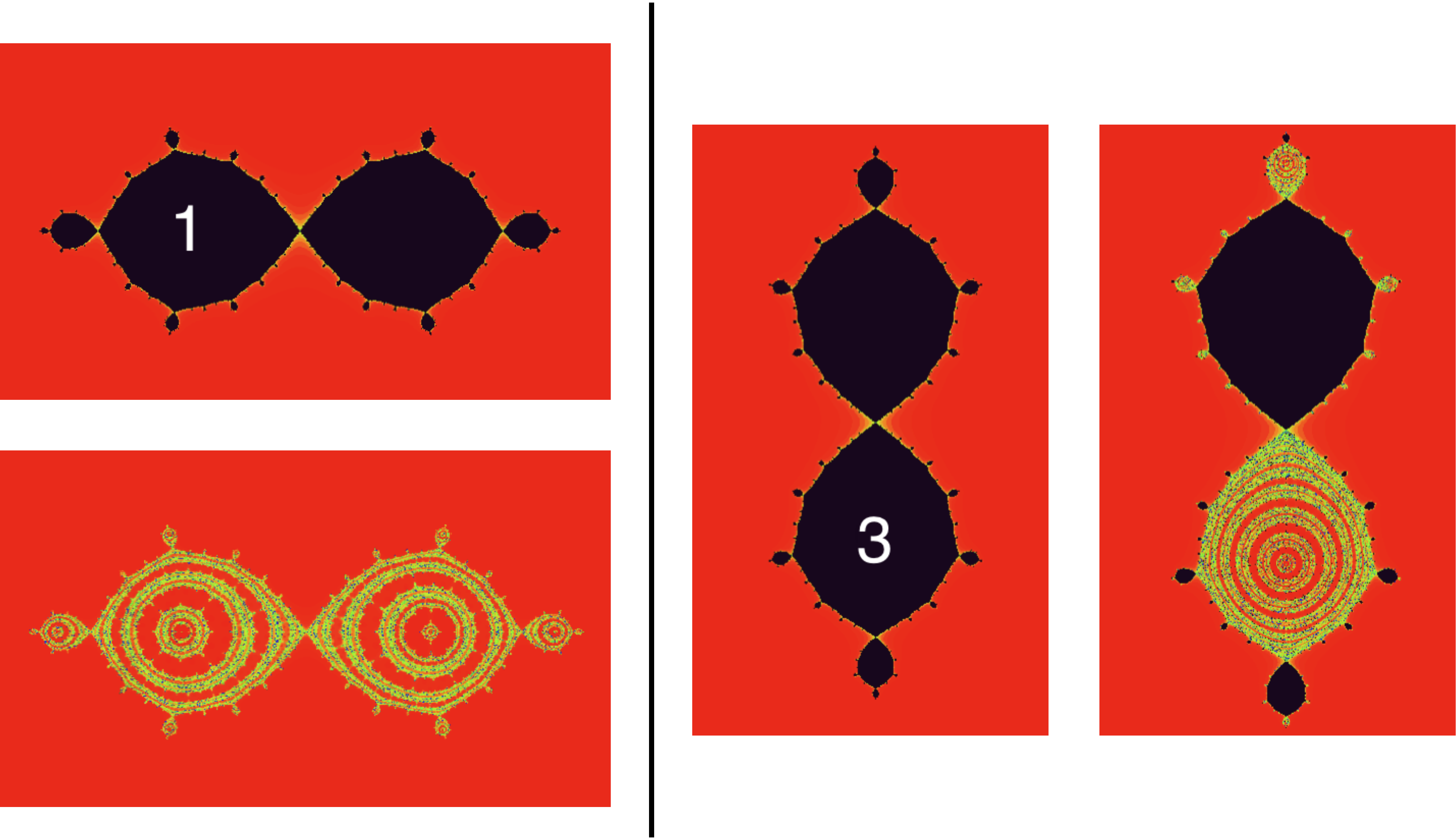}
	\setlength{\unitlength}{220pt}		  
	   \caption{\small{Two new examples of McMullen-like mappings. In the left we show the dynamical plane of the polynomial $Q(z)=2z^3-3z^2+1$ (up) and the rational map $q_{\lambda}(z)=2z^3-3z^2+1+10^{-5}/z$ (down). In the right the Milnor cubic polynomial $R_{i\sqrt{2}}(z)=z^3-i\frac{3 \sqrt{2}}{2}z^2$ (left) and the rational map $r_{\lambda}(z)=z^3-i\frac{3 \sqrt{2}}{2}z^2+ 10^{-9}/z^3$ (right). In both cases, we mark with the local degree the bounded Fatou component where we put a pole.}} 
	\label{fig:new_mcmullen}
\end{figure}

In the second example we consider the family of rational maps given by 
$g_{\lambda}(z)=z^{n}+c+\frac{\lambda}{z^{d}}$ where $c$ is a center of a hyperbolic component of the corresponding Multibrot set, thus the polynomial $P_{c}(z)=z^{n}+c$ is a $\HPcFP$. The case $n=d$ was considered previously in \cite{GeneralizedMcMullenDomain} or see  Proposition \ref{ex2:GeneralizedMcMullen}. When $n\neq d$ the only difference is the arithmetic condition (\ref{EqArithmeticCondition}) that in this case writes as $1/n+1/d<1$ and we have the following result.

\begin{pro}\label{pro:new2}
Let $g_{\lambda}(z)=z^n+c+\lambda/z^d$. For small enough values of $|\lambda|\neq 0$ and when the arithmetic condition $1/n+1/d<1$ is satisfied, the rational map $g_{\lambda}$ is a McMullen-like mapping.
\end{pro}

In the third example we also consider the polynomial $P_{c}(z)=z^n+c$, where $c$ is such that the critical orbit is periodic of period $p$. We denote by $U_{0},U_{1},\cdots,U_{p-1}$ the Fatou domains of $\Fatou(P_{c})$ containing $0,P_{c}(0),\cdots,P_{c}^{p-1}(0)$, respectively. As in \cite{SingularPerturbationsQuadraticMultiplePoles} we can consider McMullen-like mappings of type $(P_{c},\Pole)$ where the pole data $\Pole$ is formed by the Fatou domains $\{U_{0},U_{1},\cdots,U_{p-1}\}$, and the corresponding positive integers $d_{0},d_{1},\cdots,d_{p-1}$. According to Theorem A the arithmetic condition (\ref{EqArithmeticCondition}) writes as
\[
\left(\frac{1}{2}+\frac{1}{d_{0}}\right) \left( 1 + \frac{1}{d_{1}} \right)  \cdots  \left( 1 + \frac{1}{d_{p-1}} \right) \, < \, 1,
\]
\noindent since $U_{0}$ is the only bounded Fatou domain containing a critical point and we have that $P_{c}:U_{0}\to U_{1}$ is 2-to-1 while in the rest of bounded Fatou domains $P_{c}$ acts conformally.  In  \cite{SingularPerturbationsQuadraticMultiplePoles}  (see also Proposition \ref{ex3:Multipoles}) the arithmetic condition to conclude that the rational map  $h_{\lambda}(z)=z^n+c+\lambda/\prod_{j=0}^{p-1}(z-c_{j})^{d_{j}}$ is a McMullen-like mapping is $2d_{1}>d_{0}+2$ and $d_{j+1}>d_{j}+1$ for every $1\leqslant j\leqslant p-1$ (with $d_{p}=d_{0}$). We claim that these conditions imply the arithmetic condition (\ref{EqArithmeticCondition}). To see the claim, we observe that

$$
\begin{array}{ccccc}
\dfrac{1}{2}+\dfrac{1}{d_{0}} & = & \dfrac{d_{0}+2}{2d_{0}} & < & \dfrac{d_{1}}{d_{0}} \\ 
1+\dfrac{1}{d_{1}} & = & \dfrac{d_{1}+1}{d_{1}} & < & \dfrac{d_{2}}{d_{1}} \\
\cdots  & = & \cdots & < & \cdots \\
1+\dfrac{1}{d_{p-1}} & = & \dfrac{d_{p-1}+1}{d_{p-1}} & < & \dfrac{d_{0}}{d_{p-1}} \\
\end{array}
$$
So, the arithmetic condition  (\ref{EqArithmeticCondition})  
\[
\left(\frac{1}{2}+\frac{1}{d_0}\right) \left( 1 + \frac{1}{d_1} \right)  \cdots  \left( 1 + \frac{1}{d_{p-1}} \right) \, < \,  \frac{d_1}{d_0} \cdot \frac{d_2}{d_1} \cdots  \frac{d_0}{d_{p-1}}=1,
 \]
\noindent is satisfied. However the conserve is not true. Take for example $p=2$, $d_{0}=3$ and $d_{1}=6$, for these values the arithmetic condition (\ref{EqArithmeticCondition}) is $(1/2+1/3)(1+1/6)\approx0.9722<1$, however the conditions in Proposition  \ref{ex3:Multipoles} are $2d_1>d_0+2$ and $d_0 > d_1+1$ which are not satisfied.  So, we have the following result.

\begin{pro}\label{pro:new3}
Let $P_{c}(z)=z^n+c$ be a $\HPcFP$ and the pole data $\Pole$ formed by the Fatou domains $\{U_{0},U_{1},\cdots,U_{p-1}\}$ and let $d_{0},d_{1},\cdots,d_{p-1}$ be positive integers. There exists a McMullen-like mapping of type $(P_{c},\Pole)$ if and only if $ \left(\frac{1}{2}+\frac{1}{d_0}\right) \left( 1 + \frac{1}{d_1} \right)  \cdots  \left( 1 + \frac{1}{d_{p-1}} \right) \, < \, 1$.
\end{pro}

In the next example we show a McMullen-like mapping obtained from a polynomial $P$ with more that one cycle. We consider the Milnor cubic polynomials (\cite{CubicPolynomialMaps1}), this is a particular slice of the family of all the cubic polynomials fixing the behavior of one of the two critical points. More precisely, we consider the family of polynomials given by  $R_{a}(z)=z^3-\frac{3}{2}az^2$. The polynomial $R_{a}$ has a super-attracting fixed point at the origin, $R_{a}(0)=R'_{a}(0)=0$ and the other critical point is located at $z=a$. For example for  $a=i\sqrt{2}$ the polynomial $R_{i\sqrt{2}}$ has two finite super-attracting fixed points: one at $0$ and another one at $i\sqrt{2}$. We restrict to this value of $a$, however the same result is true assuming that  the polynomial $R_{a}$ is $\HPcFP$. We denote by $U_{0}$ and $U_{i\sqrt{2}}$ the Fatou domains containing $0$ and $i\sqrt{2}$, respectively, and we consider the pole data $\Pole$ formed by the Fatou domain $U_{0}$ and a positive integer $d=d_{0}$. Using the same ideas as in Proposition \ref{pro:new1} we have the following result.

\begin{pro}\label{pro:new4}
Let $r_{\lambda}(z)=z^3-i\frac{3\sqrt{2}}{2}z^2+\lambda /z^d$ with $d>2$.  For small enough values of $|\lambda|\neq 0$, the rational map $r_{\lambda}$ is a McMullen-like mapping. 
\end{pro}

We leave as an exercise to check that $r_{\lambda}$ is a McMullen-like mapping (see Proposition \ref{pro:new1})  and we only check the arithmetic condition (\ref{EqArithmeticCondition}). In this case we have two cycles, so $N=2$, and the arithmetic condition writes as
\[
 \max\{ 1/2, 1/2+1/d \} < 1\, ,
\] 
\noindent since $U_{0}\in\Pole$, $R_{i\sqrt{2}}:U_{0}\to U_{0}$ with degree two, $U_{i\sqrt{2}}\notin\Pole$, $R_{i\sqrt{2}}:U_{i\sqrt{2}}\to U_{i\sqrt{2}}$ with degree two, and $R_{i\sqrt{2}}$ acts conformally in all the rest of the bounded Fatou domains of $\Fatou(R_{i\sqrt{2}})$.  In Figure \ref{fig:new_mcmullen} we show the dynamical planes of the Milnor cubic polynomial $R_{i\sqrt{2}}$ and the McMullen-like mapping $r_{\lambda}$.

Finally, we turn our attention to the McMullen family $f_{\lambda}(z)=z^n+\lambda/z^d$. We can obtain two new examples of McMullen-like mappings topologically conjugate on their Julia set to the McMullen family (see Theorem \ref{ThmJuliaHomeomorphic}). We consider the family of rational maps given by $\widetilde{f_{\lambda}}(z)=z^n+\lambda/(z-a)^d$, in this case we first fix a small value of $|\lambda|\neq 0$ and then for $|a|\neq 0$ small enough the rational map $\widetilde{f_{\lambda}}$ is a McMullen-like mapping, where the polynomial $P$ is $z\mapsto z^n$ and the pole data $\Pole$ is also the Fatou domain $U_{0}$ of $\Fatou(P)$ containing the origin. In this case the pole is moved to $a$ and not located at the origin as in the McMullen family. The arithmetic condition (\ref{EqArithmeticCondition}) for $\widetilde{f_{\lambda}}$ is $1/n+1/d<1$, as in the McMullen family. The Julia set of $\widetilde{f_{\lambda}}$ has already been studied in \cite{EvolutionMcMullenDomain}. The same result is true if we take the family of rational maps $\widehat{f_{\lambda}}(z)=z^n+L(z)/z^d$ where $L(z)$ is a polynomial with $\deg(L) <d$. Thus, if $|L(0)|\neq 0$ is small enough the map $\widehat{f_{\lambda}}$ is a McMullen-like mapping of the same type of $f_{\lambda}$ and also with the same arithmetic condition.


\bibliographystyle{alpha}
\bibliography{biblio}
\addcontentsline{toc}{section}{References}

\end{document}